\documentclass[a4paper,12pt]{amsart}

\title{Poorly connected groups}
\author{David Hume}
\address{Mathematical Institute, University of Oxford, Oxford, OX2 6GG.}
\email{david.hume@maths.ox.ac.uk}
\author{John M. Mackay}
\address{School of Mathematics, University of Bristol, Bristol, BS8 1TX.}
\email{john.mackay@bristol.ac.uk}
\address{}
\email{}
\date{\today}
\thanks{The first author was supported by a Titchmarsh Fellowship of the University of Oxford. The second author was supported in part by EPSRC grant EP/P010245/1.}

\usepackage{amsmath,amsthm,amsfonts,graphicx,amssymb}
\usepackage{hyperref}
\usepackage{xcolor}

\numberwithin{equation}{section}
\newtheorem{theorem}[equation]{Theorem}
\newtheorem{proposition}[equation]{Proposition}
\newtheorem{corollary}[equation]{Corollary}
\newtheorem{lemma}[equation]{Lemma}

\theoremstyle{definition}

\newtheorem{question}[equation]{Question}

\newtheorem{definition}[equation]{Definition}
\newtheorem{remark}[equation]{Remark}

\newtheorem*{theorem*}{Theorem}
\newtheorem*{assump*}{Standing assumption}
\newtheorem*{remark*}{Remark}
\newtheorem*{claim*}{Claim}

\newtheoremstyle{citing}% name
  {3pt}%      Space above, empty = `usual value'
  {3pt}%      Space below
  {\itshape}% Body font
  {}%         Indent amount (empty = no indent, \parindent = para indent)
  {\bfseries}% Thm head font
  {}%        Punctuation after thm head
  {.5em}%     Space after thm head: " " = normal interword space;
  {\thmnote{#3}}% Thm head spec

\theoremstyle{citing}
\newtheorem*{varthm}{}% all text supplied in the note

\DeclareMathOperator{\diam}{diam}

\DeclareMathOperator{\cut}{cut}
\DeclareMathOperator{\sep}{sep}

\newcommand{\setcon}[2]{\left\{#1\ \left|\ #2\right.\right\}}
\newcommand{\abs}[1]{\left\lvert#1\right\rvert}

\newcommand{\ra}{\rightarrow}
\newcommand{\R}{\mathbb{R}}

\newcommand{\N}{\mathbb{N}}
\newcommand{\Z}{\mathbb{Z}}

\def\XXint#1#2#3{{\setbox0=\hbox{$#1{#2#3}{\int}$}
\vcenter{\hbox{$#2#3$}}\kern-.5\wd0}}

\numberwithin{equation}{section}

\begin{document}

\begin{abstract} We investigate groups whose Cayley graphs have poor\-ly connected subgraphs.  We prove that a finitely generated group has bounded separation in the sense of Benjamini--Schramm--Tim\'ar if and only if it is virtually free. We then prove a gap theorem for connectivity of finitely presented groups, and prove that there is no comparable theorem for all finitely generated groups. Finally, we formulate a connectivity version of the conjecture that every group of type $F$ with no Baumslag-Solitar subgroup is hyperbolic, and prove it for groups with at most quadratic Dehn function.
\end{abstract} 

\maketitle

\section{Introduction}\label{sec:intro}

When studying an infinite group through the geometry of its Cayley graphs, a natural question to ask is: If the Cayley graph is poorly connected, what does this imply about the structure of the group?

If we interpret this question as asking about disconnecting the Cayley graph by sets of finite diameter, we arrive at the theory of ends as explored by Freudenthal, Hopf, Stallings and others.
However, we can also vary the question by instead asking about disconnecting  the Cayley graph, or all its subgraphs, by sets of finite, or at least relatively small, \emph{volume}.

 The invariant we use to make this precise is the separation profile, which was introduced by Benjamini, Schramm and Tim\'ar~\cite{BenSchTim-12-separation-graphs} as a measurement of how hard it can be to cut subgraphs of $X$ into components of at most half the size.

 In this paper we study groups where the separation profile is small: we characterise those groups with bounded separation profile, find a gap theorem for finitely presented groups, and explore connections with Gromov hyperbolicity.

 We begin by recalling the definition of the separation profile.
\begin{definition}\label{def:cut-sep} 
  A subset $S$ of the vertex set $V\Gamma$ of a finite graph $\Gamma$ is a \textbf{cut--set} of $\Gamma$, if $\Gamma-S$ has no connected component with more than $\abs{\Gamma}/2$ vertices.
  The \textbf{cut--size} of $\Gamma$, $\cut(\Gamma)$, is the minimal cardinality of a cut set of $\Gamma$.
  The \textbf{separation profile} of an infinite graph $X$ is the function $\sep_X:\N\to\N$ defined by
\[
\sep_X(n)=\max\setcon{\cut(\Gamma)}{\Gamma \subset X,\ \abs{\Gamma}\leq n}.
\]
\end{definition}

  We consider separation profiles up to the equivalence $\simeq$ defined by $f\simeq g$ if $f\lesssim g$ and $g\lesssim f$, where $f\lesssim g$ if there exists a constant $C>0$ such that $f(n)\leq Cg(Cn+C)+C$ for all $n$.

  As an invariant, the separation profile enjoys the following robustness: if $X,Y$ are bounded degree graphs and $f:X\to Y$ is a Lipschitz map such that $\sup_{y\in VY}\abs{f^{-1}(y)}<\infty$, then $\sep_X \lesssim \sep_Y$. We call a map $f$ \textbf{regular} if it satisfies the above two properties.
  
  In particular, the separation profile of a finitely generated group is independent of the choice of Cayley graph, and for any finitely generated subgroup $H$ of a finitely generated group $G$ we have $\sep_H\lesssim\sep_G$.

  To give a flavour of potential separation profiles, we note that $\Z^d$ has $\sep_{\Z^d}(n) \simeq n^{(d-1)/d}$, cocompact Fuchsian groups have separation $\simeq \log n$, and (virtually) free groups have bounded separation profiles~\cite{BenSchTim-12-separation-graphs}; there are also examples of hyperbolic groups with separation $\simeq n^\alpha$ for a dense set of $\alpha \in (0,1)$~\cite{HumeMackTess}. Separation profiles are always at most linear, the case where a graph has linear separation is completely explained in \cite{HumSepExp}. The goal of this paper is to look at the other extreme.
  
First we observe that groups with bounded separation have a simple characterisation.
\begin{theorem}\label{thm:finite-sep}
A vertex transitive, bounded degree, connected graph $X$ has bounded separation if and only if $X$ is quasi-isometric to a tree.

In particular, a finitely generated group $G$ has bounded separation if and only if $G$ is virtually free.
\end{theorem}
This follows by combining work of Benjamini, Schramm and Tim\'ar with results of Kuske and Lohrey on graphs with ``bounded treewidth'', see Section~\ref{sec:bdd-sep}.
Note that the first claim of Theorem~\ref{thm:finite-sep} fails for general bounded degree graphs: as observed in \cite{BenSchTim-12-separation-graphs}, the Sierpi\'nski triangle graph has bounded separation but is not quasi-isometric to a tree.

Theorem \ref{thm:finite-sep} raises a natural question: if a group is not virtually free, how poorly connected can it be?
In the case of finitely presented groups, we find a gap in the spectrum of possible separation profiles.
 We use the notation $ B_r$ for a closed ball of radius $r$ in a metric space, or $ B_r(x)$ if the centre $x$ of the ball is important.
\begin{theorem}\label{thm:fpgap} 
  A finitely presented group $G$  which is not virtually free satisfies
\[
 \sep_G(n)\gtrsim \kappa_G(n),
\]
where $\kappa_G$ is the \textbf{inverse growth function} of the Cayley graph of $G$:
\[ \kappa_G(n)=\max\setcon{r\in\N}{\abs{ B_r}\leq n}.\]

  In particular, if $G$ is finitely presented, either
	\begin{itemize}
		\item $\sep_G(n) \simeq 1$ and $G$ is virtually free, or
		\item $\sep_G(n) \gtrsim \log n$ and $G$ is not virtually free.
	\end{itemize}
\end{theorem}
The example of cocompact Fuchsian groups shows that the $\log n$ bound is sharp.
In the special case that $G$ is assumed to be hyperbolic (and hence finitely presented), Theorem~\ref{thm:finite-sep} and the log-gap of Theorem~\ref{thm:fpgap} were shown by Benjamini--Schramm--Tim\'ar~\cite[Theorem 4.2]{BenSchTim-12-separation-graphs}.

Theorem \ref{thm:fpgap} is proven in Section~\ref{sec:loggap} by showing that in a one-ended finitely presented group it is always possible to connect annuli of bound\-ed radius, implying that to cut a ball of radius $r$ requires (at least) a set of size proportional to $r$. We then use in a crucial way the accessibility of finitely presented groups to extend the gap theorem from one-ended finitely presented groups to all finitely presented groups.
(Given this use of accessibility, one may wonder what the separation of an inaccessible group can be.)

For finitely generated groups we show that there is no gap like that of Theorem~\ref{thm:fpgap} in the possible separation profiles.

\begin{theorem}\label{thm:infpresnogap} Let $\rho:\N\to\N$ be an unbounded non-decreasing function. There is a finitely generated group $G$ % which is not virtually free 
  such that
\[
1 \not\simeq \sep_G(n) \quad\text{and}\quad \sep_G(n) \not \gtrsim \rho(n).
\]
\end{theorem}
The groups we use are the elementary amenable lacunary hyperbolic groups constructed in \cite{Olsh_Osin_Sapir}, and the key property we require of them is that they are not virtually free, but are limits of virtually free groups (see Section~\ref{sec:fg-no-gap}). We note that these are the first examples of amenable groups whose separation profile is not $n/\kappa(n)$ where $\kappa$ is the inverse growth function.

Finally, we consider the following question, to which no counterexample is currently known.  
\begin{question}\label{q:subquad}
  If $G$ is a finitely presented group, and $\sep_G(n) = o(n^{1/2})$, then must $G$ be hyperbolic?
\end{question}

As some weak evidence for this conjecture, note that such a $G$ cannot contain a subgroup isomorphic to $\Z^2$ (with separation $\simeq n^{1/2}$) or more generally a Baumslag--Solitar group (which have separation $n^{1/2}$ or $n/\log n$ by Hume--Mackay--Tessera~\cite{HumMackTes2}), and it is a well-known question whether such groups of type $F$ must necessarily be hyperbolic.

Here we present a step towards a positive answer to Question~\ref{q:subquad}.
\begin{theorem}\label{thm:fp+quadDehnimpliessqroot} Let $G$ be a finitely presented group with (exactly) quadratic Dehn function. Then there is an infinite subset $I\subseteq \N$ such that $\sep_G(n)\gtrsim n^{1/2}$ for all $n\in I$.

  Thus, if a finitely presented group $G$ has Dehn function $\lesssim n^2$, and separation function $o(n^{1/2})$, it must be hyperbolic.
\end{theorem}
The class of groups with at-most-quadratic Dehn function is rich, including: CAT(0) groups, automatic and more generally combable groups~\cite{EpsCHLPT-Wordprocessing}, and free-by-cyclic groups~\cite{BriGro-10-quad-isop-mapping-tori}. 

The main step of the proof of Theorem \ref{thm:fp+quadDehnimpliessqroot} is the following result, which may be of independent interest.

\begin{varthm}[Proposition~\ref{prop:hypcycles}.]
 Let $X$ be a connected graph. $X$ is not hyperbolic if and only if $X$ admits arbitrarily long $18$-biLipschitz embedded cyclic subgraphs.
\end{varthm}
To show this we use Papozoglou's criterion for hyperbolicity of graphs in terms of thin bigons \cite{Pap_thin_bigons}. We then prove that any diagram whose boundary is an undistorted cycle has area which is (at least) quadratic compared to its perimeter, and that its cut size is (at least) proportional to its perimeter (see Section~\ref{sec:subquad}).

\subsection*{Acknowledgements}
We are grateful to Romain Tessera for many enlightening discussions on these topics, and in particular for contributing the idea for Theorem~\ref{thm:infpresnogap}, and also thank Itai Benjamini for being a constant source of fascinating questions.

We are grateful to J\'{e}r\'{e}mie Brieussel for directing us to the reference \cite{DiekWeiss}, which led us to \cite{KusLoh05-LogicalAspects} and enabled us to simplify our original proof of Theorem~\ref{thm:finite-sep} considerably.  We also thank R\'{e}mi Coulon for the reference~\cite[Lemma 3.24]{Olsh_Osin_Sapir}.

We also thank Yves de Cornulier, Ian Agol, Derek Holt, Benjamin Steinberg, Henry Wilton, Florian Lehner, and everybody else who has contributed to the mathoverflow discussion~\cite{Hume-Q-mathoverflow-inf-pres} related to the no gap theorem for finitely generated groups. 

\section{Bounded separation}\label{sec:bdd-sep}
In this section we characterise groups with bounded separation.

\begin{varthm}[Theorem~\ref{thm:finite-sep}.]
A vertex transitive, bounded degree, connected graph $X$ has bounded separation if and only if $X$ is quasi-isometric to a tree.

In particular, a finitely generated group $G$ has bounded separation if and only if $G$ is virtually free.
\end{varthm}
\begin{proof}
Let $X$ be a vertex transitive, bounded degree, connected graph.  
By \cite[Lemma 2.3]{BenSchTim-12-separation-graphs}, if $X$ has bounded separation then all finite subgraphs of $X$ have uniformly ``bounded treewidth''.
Thus by \cite[Proof of Theorem 2.1]{BenSchTim-12-separation-graphs} (see also \cite[Theorem 3.3, Lemma 3.2]{KusLoh05-LogicalAspects}) $X$ itself has ``bounded strong treewidth'', namely there is a tree $T$ and a map $f:X\to T$ sending $VX$ to $VT$ so that if $x,y \in VX$ are adjacent then $f(x)$ and $f(y)$ are equal or adjacent in $T$, and moreover $\sup_{z \in VT} |f^{-1}(z)| < \infty$.

Using~\cite[Theorem 3.7]{KusLoh05-LogicalAspects}, this map $f:X\to T$ can be chosen so that $\sup_{z \in VT} \diam f^{-1}(z) < \infty$.
Therefore $X$ satisfies Manning's ``Bottleneck Property'' and so is quasi-isometric to a tree~\cite[Theorem 4.6]{Man05-Geom-pseudochar}.

Conversely, if $X$ is quasi-isometric to a tree it certainly has bounded separation.

Finally, a finitely generated group is quasi-isometric to a tree if and only if it is virtually free as a consequence of work of Stallings and Dunwoody, see e.g.\ \cite[Theorem 20.45]{dru-kap-18-GGT-book}.
\end{proof}

\section{A gap between constant and logarithmic separation}\label{sec:loggap}

As stated in the introduction, we claim the following gap theorem for separation.
\begin{varthm}[Theorem~\ref{thm:fpgap}.]
  A finitely presented group $G$  which is not virtually free satisfies
\[
 \sep_G(n)\gtrsim \kappa_G(n),
\]
where $\kappa_G$ is the inverse growth function $\kappa_G(n)=\max\{r\mid \abs{B_r}\leq n\}$.

	In particular, if $G$ is finitely presented, either
	\begin{itemize}
		\item $\sep_G(n) \simeq 1$ and $G$ is virtually free, or
		\item $\sep_G(n) \gtrsim \log n$ and $G$ is not virtually free.
	\end{itemize}
\end{varthm}

\begin{proof}[Proof of Theorem~\ref{thm:fpgap}]
We begin by using the accessibility of finitely presented groups to prove the theorem, assuming that it is true in the case $G$ is one-ended.

	The group $G$ is accessible so can be written as a graph of groups, where each edge group is finite and each vertex group has at most one end~\cite{Dun-85-access}.
	Each vertex group $H$ is finitely presentable: recall that a group is finitely presentable if and only if it is coarsely simply connected (e.g.~\cite[Corollary 9.55]{dru-kap-18-GGT-book}).  It follows that as $G$ is finitely presented and $H$ is a vertex group in a splitting of $G$ over finite edge groups, $H$ is finitely presentable too. 
	Also, each vertex group $H$ is undistorted in $G$, so $\kappa_H(n) \gtrsim \kappa_G(n)$.

	Now since $G$ is not virtually free, some vertex group $H$ must be one-ended, and by the discussion above it is finitely presentable and undistorted in $G$, so applying the result in the case of one-ended groups we have:
	\[
		\sep_G(n) \gtrsim \sep_H(n) \gtrsim \kappa_H(n) 
		\gtrsim \kappa_G(n).
\]
	Finally, as balls in $G$ grow at most exponentially, $\kappa_G(n) \gtrsim \log n$.
\smallskip

	It remains to show that $\sep_G(n) \gtrsim \kappa_G(n)$ when $G$ is a finitely presented, one-ended group.
This follows from the following proposition:
\begin{proposition}\label{prop-oneended-cutball}
	Let $G$ be a one-ended, finitely presented group where all relations have
	length at most $M$, and let $X$ be the corresponding Cayley graph.
	Then $\cut(B_r) \geq {r}/{400M}$, where $B_r$ denotes the ball of radius $r$
	about the identity in $X$.
\end{proposition}
We defer the proof of this proposition until later, but observe that for any $n$, if $r = \kappa_G(n)$ we have $|B_r| \leq n < |B_{r+1}| \simeq |B_r|$, so for $X$ the Cayley graph of $G$ the proposition gives us: 
\[
	\sep_X(n) \geq \cut(B_{r}) \geq \frac{r}{400M} \simeq \kappa_G(n).\qedhere
\]
\end{proof}

Before proving the proposition, we give a lemma which allows us to avoid
connected sets in $X$.
We denote open and closed $r$-neigh\-bour\-hoods of a set $V\subset X$, for $r \geq 0$, 
as $N(V,r) = \{ z \in X : d(z, V) < r\}$ and
$\overline{N}(V,r) = \{ z \in X : d(z, V) \leq r\}$, respectively.
We denote closed annuli around $V$ as $\overline{A}(V,r,R) = \overline{N}(V,R) \setminus N(V,r)$ for $0\leq r\leq R$.
\begin{lemma}\label{lem:detour-path}
	Let $X$ be the Cayley graph of a one-ended group,
	where all relations have length at most $M$.
	Let $T$ be a bounded subset of $X$ which is $8M$-coarsely connected, i.e.\ for any $x,y \in T$ there exists a chain of points $x=x_0,x_1,\ldots,x_n=y$ with each $d(x_i,x_{i+1}) \leq 8M$.
	
	Suppose we have points $x, y \in X$ with $d(x,T), d(y,T)=4M$,
	and so that $x$ and $y$ can be connected to 
	$X \setminus N(T, 8M + \frac{1}{2}\diam(T))$
	inside $\overline{A}(T, 4M, 8M + \frac{1}{2}\diam(T))$
	by paths $\gamma_x$ and $\gamma_y$,
	respectively.
	
	Then there exists a path joining $x$ to $y$ in $\overline{A}(T, M, 4M)$.
\end{lemma}
The proof of this lemma follows \cite[Lemma 6.6]{Mac-Sis-11-relhypplane} quite closely.
\begin{proof}
	Let $x', y'$ be the other endpoints of $\gamma_x, \gamma_y$, with $d(x',T), d(y',T)=8M+\frac{1}{2}\diam(T)$.
	
	As $X$ is vertex transitive and bounded degree, there exists an infinite geodesic line $\alpha: \R \ra X$ through $x'$,
	with $\alpha(0)=x'$.  
	We claim that either $\alpha |_{(-\infty, 0]}$ or $\alpha |_{[0, \infty)}$
	gives a geodesic ray from $x'$ to infinity outside $N(T, 4M)$.
	If not, we have $z, z' \in \alpha$ on either side of $x'$ with $d(z,T), d(z',T) < 4M$,
	so $d(z,z') < 8M + \diam(T)$.
	On the other hand,
	$d(z,z') = d(z,x')+d(x',z') \geq 2(d(x',T) - 4M)$,
	thus $d(x',T) < 8M + \frac{1}{2}\diam(T)$, a contradiction.
	
	Now let $\alpha_x, \alpha_y$ be the geodesic rays from $x',y'$ which do not enter $N(T,4M)$.
	Let $x'', y''$ be the last time these rays leave $N(T, 8M + \frac{1}{2}\diam(T))$.
	By one-endedness, we can join $x'', y''$ by a simple path $\beta'$
	outside $N(T, 8M + \frac{1}{2}\diam(T))$.
	
	Let $\beta_1$ be the path outside $N(T, 4M)$ which
	starts at $x$, then follows
	$\gamma_x$ to $x'$, $\alpha_x$ to $x''$, $\beta'$ to $y''$,
	$\alpha_y$ to $y'$, $\gamma_y$ to $y$.  
	Remove loops from $\beta_1$ to make it simple, keeping the same endpoints.
	
	Let $\beta_2$ be a path inside $\overline N(T, 4M)$ which
	starts at $x$, then follows a geodesic of length $4M$ to $T$,
	then follows geodesics of length $\leq 8M$ from point to point in $T$,
	then follows a geodesic of length $4M$ to $y$.
	Again, remove loops to make $\beta_2$ simple with the same endpoints.
	If having done so $\beta_2$ does not enter $N(T,M)$, then $\beta_2 \subset \overline{A}(T,M,4M)$ serves as our desired path, so we may assume that $\beta_2 \cap N(T,M) \neq \emptyset$; let $z$ be the last vertex of $\beta_2$ with $d(z,T)<M$.
	
	Together, $\beta = \beta_1 \cup \beta_2$ give a loop in $X$.
	To prove the Lemma it suffices to consider the case when $\beta_1 \cap \beta_2 = \{x,y\}$,
	i.e., this loop is simple.
	
	Since $\beta$ represents the identity in $G$, there is a van Kampen diagram 
	$D$ for $\beta$, that is, a contractible $2$-complex $D$ in the plane labelled by a combinatorial map $\varphi$ from $D$ into the Cayley $2$-complex of $G$, so that the boundary $\partial D$ of $D$ maps to $\beta$.
	In this case, as $\beta$ is simple, $D$ is a topological disc.

	Consider the function $f(\cdot):=d(\varphi(\cdot),T)$ defined on the $1$-skeleton $D^{(1)}$ of $D$, which $\varphi$ maps into $X$.
	On $\partial D$, we have $f(x)=f(y)=4M$, and $f(z)<M$ for $z \in \beta_2 \cup \partial D$ given above.
	We consider $\partial D$ as split into three subarcs, $\gamma_{xz}$ between $x$ and $z$, $\gamma_{zy}$ between $z$ and $y$, and $\gamma_{yx}$ between $y$ and $x$; note that $\varphi(\gamma_{yx})=\beta_1$ and $\varphi(\gamma_{xz}\cup\gamma_{zy})=\beta_2$.
	On $\gamma_{xz}\cup\gamma_{zy}$ we have $f \leq 4M$,
	and on $\gamma_{yx}$ we have $f \geq 4M$.

	Let $D' \subset D$ be the union of closed $2$-cells $F \subset D$ which
	have a point $u \in F \cap X$ with $f(u) \geq 2M$.
	
	Let $D''$ be the connected component of $x$ in $D'$.
	Let $\partial_O D''$ be the outer boundary path of $D''$, considering $D''$ as a subcomplex of the plane (and ignoring any bounded regions it encloses).
	Every point $p$ in $\partial_O D''$ satisfies $p\in\partial D$ or $f(p)<2M$, or both.
	
	Consider the path that follows $\partial_O D''$ from $x$ starting along $\gamma_{xz}$ and continues until it hits a point $p$ of $\gamma_{zy}\cup\gamma_{yx}$; as $\gamma_{yx} \subset \partial_O D''$ such a point exists.
	By continuity $f(p)\leq 2M$, so $p \in \gamma_{zy}$, and we can continue from $p$ along $\gamma_{zy}$ to $y$.  Along this entire path $f \in [2M-M, 4M]$, i.e.\ we have found our path in $\overline{A}(T,M,4M)$.
\end{proof}

We can now show that balls in one-ended groups are at least a little hard to cut.
\begin{proof}[Proof of Proposition~\ref{prop-oneended-cutball}]
	Let $S \subset B_r$ be given with $|S| < r/400M$.
	We will show that
	$B_r \setminus S$ must have a connected component of size $> |B_r|/2$.
	
	Let $\sim$ be the equivalence class on $S$ generated by requiring $p \sim q$ 
	if $d(p,q) \leq 8M$.
	Let $S = S_1 \sqcup \cdots \sqcup S_k$ be the decomposition of $S$ into
	equivalence classes.
	Let $V_1 = \overline{N}(S_1, 4M), \ldots, V_k = \overline{N}(S_k, 4M)$,
	and observe that $S_i$ is $8M$--coarsely connected in $V_i$.
	Note too that for $i \neq j$, $V_i \cap V_j = \emptyset$.
	
	Let $U_1, \ldots, U_k$ be given by $U_i = N(S_i, 12M + \diam(S_i)))$.
	We claim that given $p, q$ in $B_r$ outside $\bigcup_i U_i$, we can join $p$ to $q$ in $B_r \setminus S$:
	
	Consider the oriented path $\gamma_0 = [p,1] \cup [1,q]$.
	We modify $\gamma_0$ by following along $\gamma_0$ and considering each $V_{i}$
	which it meets.  Observe that every $V_{i}$ which it meets lies in $B_r$,
	for, supposing $V_i \cap [1,p] \neq \emptyset$,
	\begin{align*}
		d(1,V_i)+\diam(V_i) 
		& \leq d(1,p) - d(p,V_i) + \diam(V_i) 
		\\ & \leq d(1,p) - \left(d(p,S_i)-4M\right) +\left(\diam(S_i)+8M\right)
		\\ & \leq d(1,p) \leq r.
	  \end{align*}
	Suppose $\gamma_0$ first meets $V_{i_1}$.
	Using Lemma~\ref{lem:detour-path} applied to $T = S_{i_1}$,
	reroute $\gamma_0$ in $\overline{A}(S_i, M, 4M) \subset V_{i_1}$
	from the first	time $x$ it reaches $\overline{N}(S_i, 4M)$ to the last time $y$
	it leaves $\overline{N}(S_i, 4M)$.
	Continue for the next $V_{i_2}$ which it reaches, 
	all the way until one reaches $q$,
	and call this new path $\gamma_1$, which has our desired property:	
	since $\gamma_1$ avoids every $S_i$, it avoids $S$.

	It remains to show that $\bigcup_i U_i$ is a small set.
	Each $U_i$ lives in a ball of radius $r_i = 12M+2\diam(S_i) \leq 12M + 16M|S_i|$.
	The total diameter of these balls is
	\[
		\leq 2 \sum_i \big( 12M + 16M|S_i| \big)
		\leq 24M |S| + 32M |S| < 56 M \cdot \frac{r}{400M} \leq \frac{r}{6}.
	\]
	Take a geodesic segment $\gamma'$ in $B_r$ from $1$ of length $r$.
	We can lay out three disjoint copies of these balls along the segments
	$[0,r/6]$, $[r/3,r/2]$, $[2r/3,5r/6]$,
	and so $| \bigcup_i U_i | \leq \frac{1}{3} |B_r|$.
\end{proof}

\begin{remark}
	A variation of the proof of Theorem~\ref{thm:fpgap} shows that the bound 
	\[
	\sep_X(n) \gtrsim \kappa_X(n) := \max\{k \,|\, \exists x : |B_k(x)|\leq n\}
	\]
	holds for any one-ended, vertex transitive graph $X$ that is coarsely simply connected and of bounded degree. It is quite conceivable that the `vertex transitive' and `one-ended' assumptions can be weakened.
\end{remark}

\section{No gap for finitely generated groups}\label{sec:fg-no-gap}

Here we prove that there cannot be a gap theorem near bounded separation for finitely generated, infinitely presented groups. The key ingredient is families of epimorphisms
\[
 \langle\alpha_0,t\mid\ \rangle \to G_1 \to G_2 \to \ldots \to G_n \to \ldots G
\]
satisfying the following three properties:
\begin{itemize}
\item for each $i\in \N$, $G_i$ is virtually free,
\item $G$ is not virtually free,
\item having already fixed $\langle\alpha_0,t\mid\ \rangle \to \ldots \to G_n$ for any $r$ we may choose $G_{n+1}$ so that the homomorphism $G_n\to G_{n+1}$ is injective on balls of radius $r$ measured with respect to the generating set (the image of) $\{\alpha_0,t\}$.
\end{itemize}

Such a construction appears in \cite[Lemma $3.24$]{Olsh_Osin_Sapir}. The elementary amenable groups constructed are denoted $G(p,\mathbf{c})$ where $p$ is a prime and $\mathbf{c}$ is an infinite sequence of natural numbers which grows sufficiently quickly. The intermediate groups $G_n$ are determined uniquely by $p$ and the finite subsequence $(c_1,\ldots,c_n)$. We now show that within this collection of groups one can construct groups with unbounded but arbitrarily small separation profile.

\begin{varthm}[Theorem~\ref{thm:infpresnogap}.] 
 Let $\rho:\N\to\N$ be an unbounded non-decreasing function.
 There is a sequence $\mathbf{c}=(c_i)_{i\in\N}$ such that $G(p,\mathbf{c})$, which is not virtually free, satisfies
\[
1 \not\simeq \sep_{G(p,\mathbf{c})}(n) \quad\text{and}\quad 
  \sep_{G(p,\mathbf{c})}(n) \not \gtrsim \rho(n).
\]
\end{varthm}
\begin{proof} We will build the desired group by constructing a sequence $\mathbf{c}$ which grows sufficiently quickly. The choice of prime will not matter in our construction. Throughout we consider groups as metric spaces with respect to the generating set $\{\alpha_0,t\}$ (strictly speaking, the image of $\{\alpha_0,t\}$ in each group $G_k, G(p,\mathbf{c})$).

Fix $c_1=1$. The corresponding group $G_1$ is virtually free, so $\sep_{G_1}\leq M_1$ for some constant $M_1$. Choose $c_2$ sufficiently large for the construction \cite[Lemma 3.24]{Olsh_Osin_Sapir} and also large enough so that $G_1\to G_2$ is injective on balls of radius $2l_1$ where $\rho(l_1)\geq M_1^2$.

For each $k\geq 2$ in turn, $G_k$ is virtually free, so $\sep_{G_k}\leq M_k$ for some constant $M_k$. Choose $l_{k+1}>l_k$ so that $\rho(l_{k+1})\geq \max\{2l_k, M_k^2\}$, then choose $c_{k+1}$ sufficiently large for the construction \cite[Lemma 3.24]{Olsh_Osin_Sapir} and also large enough so that $G_k\to G_{k+1}$ is injective on balls of radius $2l_{k+1}$. We now bound $\sep_{G(p,\mathbf{c})}$.

Let $\Gamma$ be a connected subgraph of $G(p,\mathbf{c})$ with at most $l_k$ vertices, so it has diameter at most $l_k$. The map $G_k\to G(p,\mathbf{c})$ is injective on balls of radius $2l_k$ so $\Gamma$ is a connected subgraph of $G_k$. Thus
\[
 \sep_{G(p,\mathbf{c})}(l_k) = \sep_{G_k}(l_k) \leq M_k \leq \rho(l_k)^{\frac12}.
\]
Hence $\sep_{G(p,\mathbf{c})}(n) \not \gtrsim \rho(n)$. The fact that $\sep_{G(p,\mathbf{c})}(n)\not\simeq 1$ is immediate from Theorem \ref{thm:finite-sep} because $G(p,\mathbf{c})$ is not finitely presentable, and therefore not virtually free.
\end{proof}

\section{Small separation and hyperbolicity}\label{sec:subquad}
  In this section we show the following.

\begin{varthm}[Theorem~\ref{thm:fp+quadDehnimpliessqroot}] Let $G$ be a finitely presented group with (exactly) quadratic Dehn function. Then there is an infinite subset $I\subseteq \N$ such that $\sep_G(n)\gtrsim n^{1/2}$ for all $n\in I$.

  Thus, if a finitely presented group $G$ has Dehn function $\lesssim n^2$, and separation function $o(n^{1/2})$, it must be hyperbolic.
\end{varthm}

One of the main steps is the following result which may be of independent interest.

\begin{proposition}\label{prop:hypcycles}
Let $X$ be a connected graph. $X$ is hyperbolic if and only if there is some $N$ such that every $18$-bi-Lipschitz embedded cyclic subgraph in $X$ has length at most $N$.
\end{proposition}
By an $18$-bi-Lipschitz embedded cyclic subgraph of length $N$ we mean a cycle $\alpha$ in $X$ so that for any $x,y \in \alpha$, $\frac{1}{18}d_\alpha(x,y) \leq d_X(x,y) \leq d_\alpha(x,y)$, where $d_\alpha$ and $d_X$ are the distances in $\alpha$ and $X$ respectively.
\begin{proof}
  Firstly, if there exist arbitrarily long $18$-bi-Lipschitz embedded cyclic subgraphs in $X$ then it is not hyperbolic. To complete the proof we will show that any non-hyperbolic graph contains arbitrarily large $18$-biLipschitz embedded geodesic quad\-ri\-lat\-er\-als.  

We use Papazoglou's criterion for hyperbolicity of graphs, namely, a graph is hyperbolic if and only if every geodesic bigon is thin~\cite[Theorem 1.4]{Pap_thin_bigons}.

Assume $X$ is not hyperbolic, so for every $M$ there exist finite geodesics $\gamma,\gamma'$ with common endpoints such that the Hausdorff distance between them equals some $n \geq M$.  
Fix $k$ such that $d_X(\gamma(k),\gamma')=n$, swapping $\gamma, \gamma'$ if necessary.
 
Choose $l,l'$ infimal such that
\begin{equation}\label{eq:ratio-}
 \frac{l}{d_X(\gamma(k-l),\gamma')} \geq 2, \quad \frac{l'}{d_X(\gamma(k+l'),\gamma')} \geq 2.
\end{equation}
Let $\gamma_1$ be the subarc of $\gamma$ between $\gamma(k-l)$ and $\gamma(k+l')$.
Let $\beta_1$ be a geodesic from $\gamma(k-l)$ to a closest point in $\gamma'$,
and let $\beta_2$ be a geodesic from $\gamma(k+l')$ to a closest point in $\gamma'$.
Let $\gamma_2$ be the subarc of $\gamma'$ between the endpoints of $\beta_1$ and $\beta_2$.

Since the Hausdorff distance between $\gamma, \gamma'$ is $n$ we have
$l-1 < 2 d_X(\gamma(k-l+1),\gamma') < 2n$ by (\ref{eq:ratio-}), so $l \leq 2n$, and likewise $l'\leq 2n$.
As $d_X(\gamma(k),\gamma')=n$ we have $l \geq 2n/3$, else a contradiction follows from  
\[
	2d_X(\gamma(k-l),\gamma') > 2(n-2n/3) > l;
\]
 likewise $l' \geq 2n/3$.
As the lengths $\abs{\beta_1}$, $\abs{\beta_2}$ of $\beta_1,\beta_2$ satisfy $2\abs{\beta_1}\leq l, 2\abs{\beta_2}\leq l'$, we have 
\[
d_X(\beta_1,\beta_2) \geq l+l' - \abs{\beta_1} - \abs{\beta_2} \geq \frac{1}{2} (l+l') \geq \frac{2n}{3}.
\]

Now we provide a lower bound on $d_X(\gamma_1,\gamma_2)$. 
For $a \in [0,l]$ we have $d_X(\gamma(k-a),\gamma') \geq d_X(\gamma_1(k),\gamma')-a=n-a$.  On the other hand, by \eqref{eq:ratio-} we have
$d_X(\gamma(k-a),\gamma')>a/2$, so combining these cases with the similar calculation for $d_X(\gamma(k+a),\gamma')$, we find
\[
	d_X(\gamma_1,\gamma_2) \geq \min_a \max\{n-a, \frac{a}{2}\}=\frac{n}{3}.
\]

Let $\alpha$ be the quadrilateral $\gamma_1,\beta_2,\gamma_2,\beta_1$ with distance $d_\alpha$.  As $\alpha$ has length at most $12n$, if $x,y$ are in $\gamma_1,\gamma_2$, or in $\beta_1,\beta_2$, we have
\[
	d_X(x,y) \geq \frac{n}{3} \geq \frac{1}{18} d_\alpha(x,y).
\]

Suppose now $x \in \beta_1$ and $y \in \gamma_1$; reparametrize $\beta_1$ and $\gamma_1$ so that $\beta_1(0)=\gamma_1(0)$, and fix $a,b$ so that $x=\beta_1(a), y=\gamma_1(b)$, so $d_\alpha(x,y)=a+b$.
If $b\geq l$ then $d(x,y) \geq l-|\beta_1| \geq l/2 \geq n/3$, so we have the lower bound as before.
Thus we may assume $b<l$.
Let $c=d_X(x,y)$.
Suppose for a contradiction that $c < \frac{1}{8}(a+b)$.
By \eqref{eq:ratio-} applied to $y$ we have
\[
	l-b < 2d_X(\gamma(k-l+b),\gamma') \leq 2(d_X(y,x)+d_X(x,\gamma'))
	= 2(c+|\beta_1|-a).
\]
As $2|\beta_1|\leq l$, we have $-b<2c-2a$, thus
\[
	2a < b+2c < b+\frac{a+b}{4} 
	\Rightarrow a < \frac{5}{7}b
\]
so 
\[ c \geq b-a \geq \frac{2}{7}b = \frac{2b}{7b+7a}(a+b) > \frac{2b}{12b}(a+b) =\frac{1}{6}(a+b),
\]
contradicting $c<\frac{1}{8}(a+b)$.

The final case is $x \in \beta_1$ and $y \in \gamma_2$.
Parametrise $\beta_1, \gamma_2$ so that $\beta_1(0)=\gamma_2(0)$, $x=\beta_1(a), y=\gamma(b)$, and let $c=d_X(x,y)$.
If $a \geq b/2$ then as $\beta_1$ is a shortest path to $\gamma'$, 
$c \geq a = \frac{a}{3}+\frac{2a}{3} \geq \frac{1}{3}(a+b)$.
If $a < b/2$, then by the triangle inequality
$c \geq b-a=\frac{b}{3}+\frac{b/2}{3}+(\frac{b}{2}-a)\geq \frac{1}{3}(a+b)$.
\end{proof}

\begin{remark} Since a geodesic metric space is hyperbolic if and only if every $3$-biLipschitz geodesic is uniformly close to any geodesic with the same endpoints \cite[Proposition 3.2]{CordesHume-MMB}, the above proof can easily be adapted to the setting of general geodesic metric spaces again producing $N$-biLipschitz embedded cycles in any non-hyperbolic space with $N$ some universal constant.
\end{remark}

\begin{lemma}\label{lem:minarea} Let $G$ be a finitely presented group, let $A$ be a finite symmetric generating set of $G$ and let $X=\textrm{Cay}(G,A)$. For every $K\geq 1$ there exists an $L=L(K)$ such that any diagram whose boundary is a $K$-biLipschitz embedded cyclic subgraph of $X$ with boundary length $n$ has area at least $L^{-1}n^2$.
\end{lemma}
\begin{proof}
Let $\langle A\mid R\rangle$ be a finite presentation of $G$. Assume every relation in $R$ has length at most $M$ in $F(A)$.

It suffices to prove the result for all large enough $n$, so assume $n\geq 4K$.

Divide the cycle -- which we denote by $\gamma$ -- into three paths $\gamma_1,\gamma_2,\gamma_3$ of length $\lfloor\frac{n}{4}\rfloor$ and one path $\gamma_4$ of length between $\lfloor\frac{n}{4}\rfloor$ and $\lfloor\frac{n}{4}\rfloor+3$.

By construction, $d_X(\gamma_1,\gamma_3), d_X(\gamma_2,\gamma_4)\geq K^{-1} \lfloor\frac{n}{4}\rfloor$ and $\abs{\gamma_i}\geq \lfloor\frac{n}{4}\rfloor$.

Let $D$ be a van Kampen diagram with boundary $\gamma$. Let $D_1$ be the subdiagram consisting of all closed faces in $D$ containing an edge in $\gamma_1$. The closure of $\partial D_1\setminus \partial D$ in $\partial D$ contains a path $\beta_1$ connecting $\gamma_2$ to $\gamma_4$, so has length at least $K^{-1}\lfloor\frac{n}{4}\rfloor$ and is contained in the $M$-neighbourhood of $\gamma_1$. Note that $D_1$ contains at least $M^{-1}K^{-1}\lfloor\frac{n}{4}\rfloor$ faces.

Inductively, define $D_i$ to be the subdiagram consisting of all closed faces in $D$ which contain an edge in $\beta_{i-1}$ but are not in $D_{i-1}$. The closure of $\partial D_i\setminus (\partial D\cup \beta_{i-1})$ in $\partial D$ contains a path $\beta_i$ connecting $\gamma_2$ to $\gamma_4$, so has length at least $K^{-1}\lfloor\frac{n}{4}\rfloor$ and is contained in the $iM$-neighbourhood of $\gamma_1$. If $iM\leq K^{-1}\lfloor\frac{n}{4}\rfloor$ then $D_i$ contains at least $M^{-1}K^{-1}\lfloor\frac{n}{4}\rfloor$ faces.

Thus, $D$ contains at least $M^{-2}K^{-2} \lfloor\frac{n}{4}\rfloor^2$ faces. If $n\geq 4$, then $\lfloor\frac{n}{4}\rfloor\geq \frac{n}{8}$, so $D$ contains at least $L^{-1}n^2$ faces for $L=64K^2M^2$.
\end{proof}
Theorem~\ref{thm:fp+quadDehnimpliessqroot} is a consequence of the following.
\begin{lemma}\label{lem:cutlowerbd} Let $\langle A \mid R\rangle$ be a finite triangular presentation of a group $G$ such that every cycle of length at most $n$ in $X=\textrm{Cay}(G,A)$ is the boundary of a diagram with at most $Cn^2$ faces.
If $G$ is not hyperbolic then there exists some $\epsilon>0$ depending only on $C$ and an infinite subset $I\subseteq \N$ such that
\[
 \sep_X(n)\geq \epsilon n^{\frac12}, \quad \textrm{for all }n\in I.
\]
\end{lemma}
\begin{proof}
Suppose $G$ is not hyperbolic. Then by Proposition \ref{prop:hypcycles} there exists a family $\{\gamma_n\}_{n\in I}$ of $18$-biLipschitz embedded cyclic subgraphs of $X$ where the length of the cycle $\gamma_n$ is $n$, and $I$ is an infinite subset of $\mathbb{N}$.

Let $D_n$ be a minimal area van Kampen diagram with boundary $\gamma_n$; by assumption $D_n$ contains at most $Cn^2$ faces. By Lemma \ref{lem:minarea}, $D_n$ contains at least $L^{-1}n^2$ faces, where $L$ is a uniform constant.  Let $\Gamma_n$ be the $1$--skeleton of $D_n$. We claim that $\cut(\Gamma_n)\geq \epsilon \abs{\Gamma_n}^{\frac12}$, for some $\epsilon>0$ which will be determined during the proof.

Firstly, notice that $\abs{\Gamma_n}\leq 3Cn^2$.

As in the proof of Lemma \ref{lem:minarea}, split $\gamma_n$ into four subpaths $\gamma_{n,1},\gamma_{n,2},\gamma_{n,3}$ of lengths $\lfloor \frac{n}{4}\rfloor$ and $\gamma_{n,4}$ of length between $\lfloor \frac{n}{4}\rfloor$ and $\lfloor \frac{n}{4}\rfloor+3$.

Suppose that $S_n$ is a cut set of $\Gamma_n$.  List the connected components of $\Gamma_n\setminus S_n$. 
For each component $F$, the external boundary of $F$ is contained in $S_n$,
where the external boundary is the full subgraph with vertex set $\partial_e F:=\{v\in V\Gamma_n\mid d_{\Gamma_n}(v,F)=1\}$.

\textbf{Claim:} There is a cyclic graph $\gamma_F$ so that each edge of $\gamma_F$ is either in $\partial_e F$ or lies in $\partial D_n$ and has at least one endpoint in $F$.
Moreover, the subdiagram of $D_n$ with boundary $\gamma_F$ contains $F$.

\begin{proof}[Proof of Claim:]
	We have that $F$ is a full subgraph of $D_n \subset \R^2$.
	Let $D^F_n$ be the union of $F$ and all faces of $D_n$ which are bounded by $F$;
	we may assume that $\R^2\setminus D^F_n$ has no bounded components by adding any vertices and edges in such to $D^F_n$.
	Thus $D^F_n$ is contractible, and so defines a subdiagram of $D_n$.

	Because the presentation is triangular, the external boundary $\partial_e F$ of $F$ consists of a collection of paths joining points in $\partial D_n$.  By following around $\partial D^F_n$ in the plane we find a concatenation of paths in $\partial_e F$ and in $\partial D_n$ meeting $F$.
	Cutting out any loops if necessary, we find the desired cyclic subgraph.
\end{proof}

Given the claim, if $\abs{S_n}< \frac{1}{18}\lfloor \frac{n}{4}\rfloor$ then the paths in $\partial_e F \subset S_n$ for any component $F$ can meet at most two consecutive sides of $\gamma_{n,1}$, $\gamma_{n,2}$, $\gamma_{n,3}$, $\gamma_{n,4}$,
and there is exactly one connected component intersecting both $\gamma_{n,1}$ and $\gamma_{n,3}$.
This component must also intersect both $\gamma_{n,2}$ and $\gamma_{n,4}$. Every other component intersects at most $2$ consecutive sides.

All components $F$ which intersect no sides are contained in subdiagrams with combined boundary length at most $2\abs{S_n}$ (since each edge is on the boundary of at most $2$ faces) so together these diagrams have at most $4C\abs{S_n}^2$ faces and $12C\abs{S_n}^2$ vertices, as $D_n$ has minimal area.

Now consider a component $F$ which intersects either one or two consecutive sides of $\gamma_F$.  By the claim there is a cycle enclosing $F$ which can be viewed as a concatenation of paths
\[
 \alpha_1\beta_1\alpha_2\beta_2\ldots
\]
where each $\alpha_i$ is in $\partial D_n$ and each $\beta_i$ is in $\partial_e F$. Since the boundary cycle is $18$-biLipschitz embedded the combined lengths of the $\alpha_i$ is at most $18$ times the combined lengths of the $\beta_i$. Each edge in a path $\beta_i$ contributes to at most $2$ such components. Therefore all components which intersect either one or two consecutive sides are contained in subdiagrams with combined boundary length at most $(18+2)\abs{S_n}$, where the contribution of $2\abs{S_n}$ comes from paths contained in the exterior boundary of the component and the $18\abs{S_n}$ from the subpaths of $\gamma_F$.

Thus these components contain at most $3C(20\abs{S_n})^2$ vertices, as $D_n$ has minimal area. It follows that the component which intersects all four sides contains at least
\[
 v_n=\abs{\Gamma_n} - 3C(20\abs{S_n})^2 - 12C\abs{S_n}^2 - \abs{S_n}
\]
vertices. For $\epsilon'$ sufficiently small (independent of $n$) we have $\abs{S_n} < \epsilon' n$ implies $v_n> \abs{\Gamma_n}-\frac{1}{2L}n^2 \geq \frac12\abs{\Gamma_n}$, which contradicts $S_n$ being a cut set.

Since $\abs{\Gamma_n}\leq 3Cn^2$ we have $\cut(\Gamma_n) \geq \epsilon'n \geq \epsilon \abs{\Gamma_n}^{\frac12}$, for some $\epsilon>0$ independent of $n$.
\end{proof}

\begin{corollary}\label{lowerbdgeneralDehn} If $G$ is a finitely presented group with Dehn function at most $\Delta(n)$, then on some infinite subset $I$ of $\N$ we have
\[
 \sep_G(n) \gtrsim \Delta^{-1}(n):=\max\{k \mid \Delta(k)\leq n\}.
\]
\end{corollary}
\begin{proof}
  To see this simply follow the proof of Lemma \ref{lem:cutlowerbd} and deduce that if $\abs{S_n}$ is less than some small multiple of $n$ then the largest component of $\Gamma_n\setminus S_n$ is too large, so the cut size of $\Gamma_n$ is $\gtrsim n \gtrsim \Delta^{-1}(|\Gamma_n|)$.
\end{proof}

\begin{remark}
Corollary \ref{lowerbdgeneralDehn} cannot be improved simply by taking diagrams with quasi-isometrically embedded boundary and larger area, since the improvements to the cut size cancel out the increased area. For example, if the $D_n$ have cubic area in a graph with cubic Dehn function, then we can increase the size of $S_n$ to some multiple of $n$, but still get the lower bound of cube root for the separation of this diagram.
\end{remark}

\newcommand{\etalchar}[1]{$^{#1}$}
\def\cprime{$'$}

\end{document}